\documentclass{article}
\usepackage{amsmath,amssymb, amsthm, latexsym,graphicx, mathrsfs, mathtools}

\newcommand{\NN}{{\mathbb N}}

\newcommand{\RR}{{\mathbb R}}

\newcommand{\abs}[1]{ \left| #1 \right|}
\newcommand{\bbar}[1]{\overline{#1}}

\newcommand{\del}{\partial}

\newcommand{\dist}{\on{dist}}

\newcommand{\half}{\frac{1}{2}}

\newcommand{\ip}[1]{\langle #1 \rangle}

\newcommand{\nv}{^{-1}}

\newcommand{\nor}[2]{\left\|#1\right\|_{#2}}
\newcommand{\oline}[1]{\overline{#1}}
\newcommand{\oo}{\infty}
\newcommand{\pars}[1]{\left(#1\right)}

\newcommand{\uline}[1]{\underline{#1}}

\newcommand{\mcl}{\mathcal}
\newcommand{\mbb}{\mathbb}
\newcommand{\mbf}{\mathbf}
\newcommand{\on}{\operatorname}

\newtheorem{theorem}{Theorem}
\newtheorem{lemma}{Lemma}
\newtheorem{definition}{Definition}

\numberwithin{equation}{section}
\numberwithin{lemma}{section}
\numberwithin{theorem}{section}
\numberwithin{definition}{section}

\begin{document}
\title{Perron's method for pathwise viscosity solutions\thanks{Partially supported by the National Science Foundation (NSF) Research Training Group (RTG) grant DMS1246999, Panagiotis Souganidis's NSF grants DMS-1266383 and DMS-1600129, and Souganidis's Office of Naval Research (ONR) grant N000141712095.}}
\author{Benjamin Seeger}
\date{\today}

\maketitle

\begin{abstract}
	We use Perron's method to construct viscosity solutions of fully nonlinear degenerate parabolic pathwise (rough) partial differential equations. This provides an intrinsic method for proving the existence of solutions that relies only on a comparison principle, rather than considering equations driven by smooth approximating paths. The result covers the case of multidimensional geometric rough path noise, where the noise coefficients depend nontrivially on space and on the gradient of the solution. Also included in this note is a discussion of the comparison principle and a summary of the pathwise equations for which one has been proved.
\end{abstract}

\section{Introduction}

This paper is concerned with the construction of solutions of certain fully nonlinear pathwise (rough, stochastic, etc.) partial differential equations. For an initial value $u_0 \in BUC(\RR^d)$, the space of bounded, uniformly continuous functions on $\RR^d$, a finite horizon $T >0$, a continuous path $W = (W^1,W^2,\ldots, W^m): [0,T] \to \RR^m$, and functions $F: S^d \times \RR^d \times \RR \times \RR^d \times [0,T] \to \RR$ and $H = (H^1,H^2,\ldots,H^m) : \RR^d \times \RR^d \to \RR^m$, where $S^d$ is the space of $d\times d$ symmetric matrices, we consider the initial value problem
\begin{equation}\label{E:eq}
	du = F(D^2 u, Du, u, x, t)\;dt + \sum_{i=1}^m H^i(Du, x) \cdot dW^i \quad \text{in } \RR^d \times (0,T], \qquad u(\cdot,0) = u_0 \quad \text{on } \RR^d.
\end{equation}
When $W$ is continuously differentiable, $dW$ stands for $\frac{d}{dt} W(t) = \dot W(t)$ and ``$\cdot$'' denotes multiplication. In this case, the Crandall-Lions theory of viscosity solutions provides the general framework to study \eqref{E:eq}; see Crandall, Ishii, and Lions \cite{CIL}. The same notation is used when $W$ is of bounded variation, and then \eqref{E:eq} can be analyzed as in the work of Ishii \cite{I} and Lions and Perthame \cite{LP} when $F \equiv 0$, or of Nunziante \cite{N} for second-order equations. Henceforth, these situations are referred to as the ``classical viscosity'' setting.

The problem is more complicated when $W$ is a sample path of a stochastic process, such as Brownian motion. In this case, $W$ is nowhere differentiable, and in fact has unbounded variation on every interval. The symbol ``$\cdot$'' is then regarded as the Stratonovich differential. More generally, $W$ may be a geometric rough path, a specific example being Brownian motion enhanced with its Stratonovich iterated integrals, and in some situations, $W$ is even allowed to be an arbitrary continuous path. At the very least, certain differentiable equations driven by $W$, namely, the characteristic equations corresponding to the first-order part of \eqref{E:eq}, are required to have a stable pathwise theory. More details about this point are given in Sections \ref{S:assumptions} and \ref{S:prelim}.

The notion of pathwise viscosity solutions for equations like \eqref{E:eq} was developed by Lions and Souganidis, first for Hamiltonians depending smoothly on the gradient $Du$ \cite{LS1}, and later for nonsmooth Hamiltonians \cite{LS2}. The comparison principle was proved in \cite{LS3}, and in \cite{LS4}, equations with semilinear noise dependence were considered, that is, Hamiltonians depending linearly on $Du$ and nonlinearly on $u$. The theory has since been extended to treat Hamiltonians with spatial dependence, as by Friz, Gassiat, Lions, and Souganidis \cite{FGLS}, or by the author \cite{Se}; these papers use techniques developed by Lions and Souganidis that appear in forthcoming works. Many more details and results are summarized in the notes of Souganidis \cite{Snotes}. The case in which $H$ depends linearly on the gradient has been studied extensively from the point of view of rough path theory by many authors, including, but not limited to, Caruana, Friz, and Oberhauser \cite{CFO} and Gubinelli, Tindel, and Torrecilla \cite{GTT}. The problem was also examined by Buckdahn and Ma \cite{BM1, BM2} using the pathwise control interpretation.

An important feature of the pathwise viscosity theory is the stability of \eqref{E:eq} with respect to the path in a suitable topology. This leads naturally to a notion of weak solutions. More precisely, it has been shown in various situations that, for suitable $C^1$-families $\{W^\eta\}_{\eta > 0}$ that satisfy $\lim_{\eta \to 0} W^\eta = W$ in an appropriate sense, if $u^\eta$ is the classical viscosity solution of the initial value problem
\begin{equation}\label{E:smootheq}
	u^\eta_t = F(D^2 u^\eta, Du^\eta, u^\eta, x, t) + \sum_{i=1}^m H^i(Du^\eta, x) \dot W^{i,\eta}(t) \quad \text{in } \RR^d \times (0,T], \qquad u^\eta(\cdot,0) = u_0 \quad \text{on }\RR^d,
\end{equation}
then, as $\eta \to 0$, $u^\eta$ converges uniformly to a unique limit $u$, independently of the approximating family $\{W^\eta\}_{\eta > 0}$. 

On the other hand, it is of interest to develop a solution theory for \eqref{E:eq} that is intrinsic to the equation. When the $H^i$ are linear in $Du$, this can be accomplished by using a ``flow transformation'' to eliminate the irregular terms involving $dW^i$, leading to an equation that can be analyzed through classical means. However, if $H^i$ is nonlinear, then no such global transformation exists, and the correct procedure for defining pathwise viscosity sub- and super-solutions is to cancel out the rough noise with the use of suitable test functions. It is this approach that was developed in \cite{LS1}, and the one we use throughout this paper (for the relevant definitions, see Section \ref{S:prelim}). 

The uniqueness for pathwise solutions of \eqref{E:eq} defined in this way is then proved by establishing a comparison principle. That is, if $u$ and $v$ are respectively an upper- and lower-semicontinuous sub- and super-solution, then, for all $t \in (0,T]$, 
\[
	\sup_{\RR^d} (u(\cdot,t) - v(\cdot,t))_+ \le \sup_{\RR^d} (u(\cdot,0) - v(\cdot,0))_+.
\]
The comparison principle immediately implies the following:
\begin{theorem} \label{T:perrons}
	There exists at most one solution of \eqref{E:eq} in $BUC(\RR^d \times [0,T])$, which is given by the maximal sub-solution (or minimal super-solution).
\end{theorem}

In other words, if a pathwise viscosity solution of \eqref{E:eq} exists, then it is characterized by the formula
\begin{equation}
	u(x,t) := \sup\left\{ v(x,t) : \text{ $v$ is a sub-solution of \eqref{E:eq}} \right\}. \label{E:solution}
\end{equation}

Our objective is to use the comparison principle to prove directly that \eqref{E:solution} defines the unique solution of \eqref{E:eq}, thereby completing the program of developing a solution theory for \eqref{E:eq} that avoids consideration of equations with smooth paths. As in the classical theory for nonlinear elliptic and parabolic equations, we accomplish this by proving a different kind of stability for \eqref{E:eq}, namely, the stability of sub-solutions under taking suprema. We also construct suitable sub- and super-solutions of \eqref{E:eq} to ensure that \eqref{E:solution} is actually a super-solution, and that it achieves the desired initial data. The strategies for doing so resemble those from Section 4 of \cite{CIL}, but a more involved analysis is required, due to the rough nature of the test functions.
 
\subsection*{Organization of the paper} The main assumptions for $F$, $H$, and $W$ are given in Section \ref{S:assumptions}. Section \ref{S:prelim} contains definitions and some preliminary remarks regarding pathwise viscosity solutions. In Section \ref{S:steps}, the main steps of the Perron construction and the proof of Theorem \ref{T:perrons} are presented. Finally, examples of Hamiltonians and paths for which \eqref{E:eq} satisfies the comparison principle are given in the Appendix.

\subsection*{Notation} $S^d$ is the space of symmetric $d$-by-$d$ matrices, $I_d \in S^d$ is the identity matrix, and, for $X,Y \in S^d$, the inequality $X \le Y$ means that $\ip{X\xi, \xi} \le \ip{Y\xi, \xi}$ for all $\xi \in \RR^d$, where $\ip{\cdot, \cdot}$ denotes the usual inner product on $\RR^d$. If $B$ is an arbitrary $m$-by-$m$ matrix, then $\text{Sym}(B) \in S^m$ is defined by $\text{Sym}(B) := \frac{1}{2} (B B^t + B^t B)$, where $B^t$ is the transpose of $B$. 

The positive and negative part of a number $x \in \RR$ are denoted respectively by $x_+ := \max(x,0)$ and $x_- := \max(-x, 0)$. If $K \subset \RR^d$ and $r > 0$, then the set of points in the interior of $K$ that are a fixed distance $r$ from the boundary is given by $K_r := \{ x \in \RR^d : \dist(x,K^c) \ge r\}$. Open balls and cylinders in $\RR^d$ and $\RR^d \times [0,T]$ are expressed, for $r, s, t_0 > 0$ and $x_0 \in \RR^d$, by $B_r(x_0) := \{ x \in \RR^d : |x - x_0| < r\}$, $N_{r,s}(x_0,t_0) := B_r(x_0) \times (t_0 - s, t_0 +s)$, and $N_r(x_0,t_0) := N_{r,r}(x_0,t_0)$. Intervals of the form $(t_0 - h, t_0 + h)$ are assumed to mean $(t_0 - h, t_0 +h) \cap [0,T]$. 

$C^k_b(\Omega)$ is the space of functions with bounded and continuous derivatives through order $k$ on some $\Omega \subset \RR^N$, while $BUC(\RR^d \times [0,T])$, $USC(\RR^d \times [0,T])$, and $LSC(\RR^d \times [0,T])$ are respectively the spaces of bounded uniformly continuous, upper-semicontinuous, and lower-semicontinuous functions on $\RR^d \times [0,T]$. For $U: \RR^d \times [0,T] \to \RR$, the upper-semicontinuous and lower-semicontinuous envelopes $U^*$ and $U_*$ of $U$ are defined respectively as
\[
	U^*(x,t) := \limsup_{(y,s) \to (x,t)} U(y,s) \quad \text{and} \quad
	U_*(x,t) := \liminf_{(y,s) \to (x,t)} U(y,s).
\]
	
\section{Assumptions} \label{S:assumptions}

Throughout the paper, it is assumed that $F: S^d \times \RR^d \times \RR \times \RR^d \times [0,T] \to \RR$ is continuous, bounded for bounded $(X,p,r) \in S^d \times \RR^d \times \RR$, degenerate elliptic, and nonincreasing in $r \in \RR$; that is,
\begin{equation} \label{A:Fassumptions}
	\left\{
	\begin{split}
		&F \in C(S^d \times \RR^d \times \RR \times \RR^d \times [0,T]) \cap C_b(B_R(0) \times \RR^d \times [0,T]) \quad \text{for all $R >0$,} \\[1.2mm]
		&X \mapsto F(X,\cdot,\cdot,\cdot,\cdot) \text{ is nondecreasing, and } r \mapsto F(\cdot,\cdot,r,\cdot,\cdot) \text{ is nonincreasing.}  
	\end{split}
	\right.
\end{equation}

The Hamiltonians require more regularity than $F$, namely
\begin{equation}\label{A:Hregularity}
	\left\{
	\begin{split}
		&H \in C^2_b(B_R(0) \times \RR^d; \RR) \quad \text{for all $R>0$} \quad \text{if $m = 1$, and} \\[1.2mm]
		&H \in C^4_b(\RR^d \times \RR^d; \RR^m) \quad \text{otherwise.}
	\end{split}
	\right.
	\end{equation}
The different assumptions for $m = 1$ and $m > 1$ are related to the interpretation of the path $W$, for which we assume that
\begin{equation}\label{A:pathassumption}
	\left\{
	\begin{split}
		&W(0) = 0, \quad W \in C([0,T], \RR^m), \quad \text{and, if $m > 1$,}  \\[1.2mm]
		&\text{for some $\mbb W: [0,T] \times [0,T] \to \RR^m \otimes \RR^m$ and $\alpha \in \left( \frac13, \frac12 \right]$,} \\[1.2mm]
		&\mbf W := (W, \mbb W) \in \mathscr{C}^\alpha_g([0,T], \RR^m). 
	\end{split}
	\right.
\end{equation}
Here, $\mathscr{C}^\alpha_g$ is the space of $\alpha$-H\"older continuous geometric rough paths; that is, $\mbf W \in \mathscr{C}^\alpha_g([0,T], \RR^m)$ if
\begin{equation} \label{E:roughpathdef}
	\left\{
	\begin{split}
		&\nor{\mbf W}{\mathscr{C}^\alpha} := \sup_{s\ne t} \frac{\abs{W_s - W_t}}{|s-t|^\alpha} + \sup_{s \ne t} \frac{\abs{ \mbb W_{st}}}{|s-t|^{2\alpha}} < \oo,  \\[1.2mm]
		&\mbb W_{st} - \mbb W_{su} - \mbb W_{ut} = (W_u - W_s) \otimes (W_t - W_u) \quad \text{for any $s,u,t \in [0,T]$, and} &\\[1.2mm]
		&\text{Sym}(\mbb W_{st}) = \frac{1}{2} (W_t - W_s) \otimes (W_t - W_s) \quad \text{for any $s,t \in [0,T]$.}
	\end{split}
	\right.
\end{equation}
The quantity $\nor{\mbf W}{\mathscr{C}^\alpha}$ is called the rough-path norm of $\mbf W$, although $\mathscr{C}^\alpha_g$ is not a linear space, due to the nonlinear nature of the second two constraints in \eqref{E:roughpathdef}.

If $W$ is smooth, then $\mbb W$ is automatically given by the Riemann-Stieltjes integrals $\mbb W^{ij}_{s,t} := \int_s^t (W^i_r - W^i_s)\; dW^j_r$ for $1 \le i , j, \le m$. In view of the fact that Brownian paths are $\alpha$-H\"older continuous for any $\alpha \in \pars{0, \frac{1}{2}}$, a Brownian motion can be viewed as a random geometric rough path with $\mbb W$ given by the iterated Stratonovich integrals
\[
	\mbb W^{ij}_{s,t} = \int_s^t \pars{ W^i_r - W^i_s} \circ dW^j_r.
\]
The third restriction in \eqref{E:roughpathdef} ensures that geometric rough paths satisfy, formally, the standard chain and product rules from differential calculus. Observe, therefore, that while the first two properties in \eqref{E:roughpathdef} are also satisfied by the corresponding iterated It\^o integrals for Brownian motion, the third property is not, and thus, the Brownian rough path for which $\mbb W$ is defined through the It\^o differential does not belong to $\mathscr{C}^\alpha_g$.  

A further characteristic of the space of geometric rough paths is its stability with respect to appropriate regularizations. More precisely, for any $\mbf W \in \mathscr C^\alpha_g([0,T], \RR^m)$, there exists a sequence of smooth paths $W^n : [0,T] \to \RR^n$ such that, as $n \to \oo$, $W^n$ and $\mbb W^n$ converge uniformly to respectively $W$ and $\mbb W$, and, furthermore, $\sup_{n \in \NN} \nor{\mbf W^n}{\mathscr{C}^\alpha} < \oo$. 

The results in this paper could be adapted to treat rough paths with less H\"older regularity than in \eqref{A:pathassumption}, in which case more iterated integrals are involved in the definition \eqref{E:roughpathdef}, and more regularity is required for $H$. For more details on the theory of rough paths and rough differential equations, see Friz and Hairer \cite{FH} or Friz and Victoir \cite{FV}. 

As discussed in the introduction, Perron's method is used to construct solutions of equations for which there is a notion of sub- and super-solutions, which, in turn, satisfy an appropriate comparison principle. The definition of pathwise viscosity sub- and super-solutions for \eqref{E:eq} is discussed in the next section, and it is assumed that
\begin{equation} \label{A:comparison}
	\left\{
	\begin{split}
	&\text{if $u \in USC(\RR^d \times [0,T])$ and $v \in LSC(\RR^d \times [0,T])$ are respectively} \\[1.2mm]
	&\text{a sub- and super-solution of \eqref{E:eq}, then, for all $t \in (0,T]$,} \\[1.2mm]
	&\sup_{x \in \RR^d} \pars{ u(x,t) - v(x,t)}_+ \le \sup_{x \in \RR^d} \pars{ u(x,0) - v(x,0)}_+. 
	\end{split}
	\right.
\end{equation}

By invoking \eqref{A:comparison}, we implicitly assume that $F$, $H$, and $W$ satisfy extra conditions that allow one to prove the comparison principle. The standard assumption for $F$, which comes from the classical viscosity theory, is that
\begin{equation} \label{A:Fstructure}
	\left\{
	\begin{split}
	&\text{there exist $C > 0$ and $\omega: [0,\oo) \to [0,\oo)$ satisfying $\lim_{s \to 0+} \omega(s) = 0$ such that, whenever}\\[1.2mm]
	&\text{$\lambda > 0$, $X,Y \in S^d$, and} \quad
	-C\lambda
	\begin{pmatrix}
		I_d & 0 \\
		0 & I_d
	\end{pmatrix}
	\le
	\begin{pmatrix}
		X & 0 \\
		0 & -Y
	\end{pmatrix}
	\le
	C\lambda
	\begin{pmatrix}
		I_d & -I_d \\
		-I_d & I_d
	\end{pmatrix},
	\\[1.2mm]
	&\text{then, for all $x,y \in \RR^d$, $r \in \RR$, and $t \in [0,T]$, }&\\[1.2mm]
	&F(Y,\lambda(x-y), r, y,t) - F(X,\lambda(x-y), r,x,t) \le \omega\pars{ \lambda|x-y|^2 + |x-y|}. 
	\end{split}
	\right.
\end{equation}
Notice that the matrix inequality in \eqref{A:Fstructure} implies that $X \le Y$. Therefore, in view of the ellipticity of $F$, \eqref{A:Fstructure} holds whenever $F$ is independent of $x$, or if, say, $F$ is of the form
\[
	F(X,p,r,x,t) = F_1(X,p,r,t) + F_0(p,x,r,t)
\]
where $F_0$ and $F_1$ satisfy \eqref{A:Fassumptions} and, for some $\omega$ as in \eqref{A:Fstructure},
\[
	\abs{ F_0(p,x,r,t) - F_0(p,y,r,t)} \le \omega((1+|p|)|x-y|) \quad \text{for all} \quad (p,x,y,r,t) \in \RR^{3d} \times \RR \times [0,T].
\]
A special case with more involved spatial dependence is
\[
	F(X,x) := \sum_{i,j,k=1}^d \sigma_{ik}(x) \sigma_{jk}(x) X_{ij},
\]
where $\sigma \in C^{0,1}(\RR^d, \RR^{d \times d})$. Many more examples can be constructed by observing that, if $\{F_{\alpha \beta}\}_{\alpha \in A, \beta \in B}$ is a family of functions satisfying \eqref{A:Fassumptions} and \eqref{A:Fstructure} with the same bounds and modulus $\omega$, then so are
\[
	 \inf_{\alpha \in A} \sup_{\beta \in B} F_{\alpha \beta} \quad \text{and} \quad \sup_{\beta \in B} \inf_{\alpha \in A} F_{\alpha \beta}.
\]
For more discussion on the comparison principle for degenerate second-order equations, see \cite{CIL} and the references therein.

Depending on the setting, the proof of the comparison principle for \eqref{E:eq} may require even more regularity for the Hamiltonians than \eqref{A:Hregularity}, or extra structural assumptions, such as polynomial growth for the derivatives of $H$, or uniform convexity in the $Du$ variable. In some of these cases, $W$ may also need more regularity, even if $m = 1$. More details are provided in the Appendix. The only assumptions we directly use for the Perron construction, besides the comparison principle \eqref{A:comparison}, are \eqref{A:Fassumptions}, \eqref{A:Hregularity}, and \eqref{A:pathassumption}.

\section{Preliminary remarks and the definition of sub- and super-solutions} \label{S:prelim}

\subsection{The characteristic equations} The theory of pathwise viscosity solutions depends strongly on the properties of the system of characteristic equations for the first-order part of \eqref{E:eq}, which, for $p, x \in \RR^d$ and $t_0 \in [0,T]$, is given by
\begin{equation}
	\begin{dcases}
		dX = - \sum_{i=1}^m D_p H^i(P,X)\cdot dW^i,  & X(x,p,t_0) = x, \\
		dP = \sum_{i=1}^m D_x H^i(P,X) \cdot dW^i, &  P(x,p,t_0) = p.
	\end{dcases} \label{E:chars}
\end{equation}
It follows from \eqref{A:Hregularity} that
\begin{equation}
	\text{\eqref{E:chars} has a unique solution $(X,P) \in C([0,T] ; C^1(\RR^d \times \RR^d))$.} \label{E:charsol}
\end{equation}
In the rough path setting, \eqref{E:charsol} is a consequence of the existence and uniqueness for solutions of rough differential equations and the differentiability of flows (see \cite{FH, FV}). When $m = 1$, the solution $(X,P)$ is given by $(X,P)(x,p,t) = (\tilde X, \tilde P)(x,p, W_t - W_{t_0})$, where $(\tilde X, \tilde P)$ solves the time-homogenous system
\begin{equation}
	\begin{dcases}
		\frac{d \tilde X}{dt} = - D_p H(\tilde P,\tilde X), & \tilde X(x,p,0) = x, \\
		\frac{d \tilde P}{dt} = D_x H(\tilde P, \tilde X), & \tilde P(x,p,0) = p.
	\end{dcases}
	\label{E:homogchars}
\end{equation}
The classical theory of ordinary differential equations then yields the existence, uniqueness, and regularity of $(\tilde X,\tilde P)$.

When $m > 1$, it is not possible in general to solve \eqref{E:chars} by reducing the problem to systems like \eqref{E:homogchars} with a change of variables. The exception is when the Poisson brackets of the $H^i$ vanish; that is, for all $i, j = 1,2,\ldots,m$,
\begin{equation}
	\{ H^i, H^j \} := \sum_{k=1}^d \pars{ \frac{\del H^i}{\del x_k} \frac{\del H^j}{\del p_k} -  \frac{\del H^i}{\del p_k} \frac{\del H^j}{\del x_k}} = 0. \label{A:Poisson}
\end{equation}
In this case, the Hamiltonian flows for each $H^i$ commute, and \eqref{E:chars} can be solved by composing the various Hamiltonian solution operators with the corresponding increments $W^i_t - W^i_{t_0}$. As a result, as long as \eqref{A:Poisson} holds, $H$ may be assumed to be $C^2$. A particular example is when each $H^i$ is independent of $x$, in which case $X$ takes the explicit form
\[
	X(x,p,t) = x - \sum_{i=1}^m DH^i(p) \pars{ W^i(t) - W^i(t_0)}.
\]

\subsection{Local-in-time spatially smooth solutions: the solution operator $S(t,t_0)$} 
By inverting the characteristics, it is possible to construct local-in-time, spatially-smooth solutions of the Hamilton-Jacobi part of \eqref{E:eq}:
\begin{equation} \label{E:stochHam}
	d\Phi = \sum_{i=1}^m H^i(D\Phi,x) \cdot dW^i \quad \text{in } \RR^d \times (t_0 - h, t_0 + h), \qquad
	\Phi(\cdot,t_0) = \phi \in C^2_b(\RR^d) \quad \text{on } \RR^d.
\end{equation}
The approach taken here, which is to use the properties of rough flows of diffeomorphisms coming from the theory of rough differential equations, is closely related to the stochastic flows considered by Kunita \cite{K}.

When $m > 1$, \eqref{E:stochHam} is interpreted in the rough path sense; that is, for all $(x,t) \in \RR^d \times (t_0 - h, t_0 + h)$, $\Phi$ is given by the rough integral
\[
	\Phi(x,t) = \phi(x) + \sum_{i=1}^m \int_{t_0}^t H^i(D\Phi(x,s),x) \cdot dW^i_s.
\]
If $m = 1$ and $W$ is an arbitrary continuous path, then $\Phi(x,t) := \tilde \Phi(x,W_t - W_{t_0})$, where, for some $\tau > 0$, $\tilde \Phi$ is a smooth solution of the classical equation
\[
	\tilde \Phi_{t} = H(D \tilde \Phi,x) \quad \text{in } \RR^d \times (-\tau,\tau), \qquad \tilde \Phi(\cdot,0) = \phi \quad \text{on } \RR^d,
\]
as long as $h$ is small enough that
\[
	\sup_{|t - t_0| < h} |W_t - W_{t_0}| < \tau.
\]

Fix $\phi \in C^2_b(\RR^d)$ and $t_0 \in [0,T]$, and set
\begin{equation}\label{E:solutionchars}
	\left\{
	\begin{split}
	&\mbf X(x,t) = X(x,D\phi(x),t), \quad \mbf P(x,t) = P(x,D\phi(x),t), \quad \text{and} \\[1.2mm]
	&\mbf Z(x,t) = \phi(x) + 
	\begin{dcases} 
		\int_0^{W_t - W_{t_0} } \pars{ H(\mbf P, \mbf X) - \mbf P \cdot D_p H(\mbf P, \mbf X) }\;ds & \text{if } m = 1, \text{ and} \\
		\sum_{i=1}^m \int_{t_0}^t \pars{ H^i(\mbf P, \mbf X) - \mbf P \cdot D_p H^i(\mbf P, \mbf X) }\cdot dW^i & \text{if } m > 1,
	\end{dcases}
	\end{split}
	\right.
\end{equation}
where the second expression in the definition of $\mbf Z$ is a rough path integral. The boundedness of $D^2\phi$ and the differentiability of $\mbf X$ in $x$ yield that
\begin{equation} \label{E:inversechar}
	\left\{
	\begin{split}
	&\text{there exists $h > 0$ depending only on $\nor{\phi}{C^2_b}$, the derivatives of $H$, and $\nor{\mbf W}{\mathscr{C}^\alpha}$ such that,}\\[1.2mm]
	&\text{for all $t \in (t_0 - h, t_0 + h)$, $x \mapsto \mbf X(x,t)$ is invertible on $\RR^d$, and both $t \mapsto \mbf X(\cdot,t)$}  \\[1.2mm]
	&\text{and $t \mapsto \mbf X\nv(\cdot,t)$ belong to $C((t_0 - h, t_0 +h) ; C^1_b(\RR^d))$.} 
	\end{split} 
	\right.
\end{equation}
Finally, define
\begin{equation}
	\Phi(x,t) = S(t,t_0)\phi(x) := \mbf Z(\mbf X\nv(x,t),t). \label{E:solutionoperator}
\end{equation}

\begin{lemma} \label{L:smoothsolution}
The function $\Phi$ defined by \eqref{E:solutionoperator} belongs to $C((t_0 - h, t_0+h) ; C^2_b(\RR^d))$, and is a solution of the pathwise Hamilton-Jacobi equation \eqref{E:stochHam}.
\end{lemma}

\begin{proof}
	When $m = 1$, the claim follows from the change of variables in time and classical results on Hamilton-Jacobi equations, so only the rough path setting is considered. In view of \eqref{E:charsol}, the quantities $D_x \mbf X$ and $D_x \mbf P$ are rough paths solving the rough differential equations corresponding to those obtained by differentiating \eqref{E:chars} in $x$. From this, a straightforward calculation yields
	\[
		d\pars{ D_x \mbf Z - \mbf P \cdot D_x \mbf X} = 0,
	\]
and therefore
	\[
		D\Phi(\mbf X,t)\cdot D_x \mbf X = D_x \mbf Z = P \cdot D_x \mbf X.
	\]
	It follows that $\mbf P(x,t) = D\Phi(\mbf X(x,t),t)$ for all $t \in (t_0 - h, t_0 + h)$, and so, by \eqref{E:inversechar}, $\Phi$ maps $(t_0-h,t_0+h)$ continuously into $C^2_b(\RR^d)$. It is then standard to verify that $\Phi$ solves \eqref{E:stochHam}.
\end{proof}

The next lemma summarizes some properties of the solution operators $S(t,t_0)$. The proofs are immediate or follow from the classical case by approximating $W$ with smooth paths and passing to the limit. Indeed, if $\{W^n \}_{n=1}^\oo$ is a sequence of smooth paths that converge, as $n \to \oo$, to $W$ in the rough path topology, as explained in Section \ref{S:assumptions}, then the stability of the system \eqref{E:chars} with respect to the rough path norm yields $h > 0$ independent of $n$ such that the classical solution $\Phi^n$ to \eqref{E:stochHam} driven by $W^n$ is smooth on $\RR^d \times (t_0 - h, t_0 + h)$, and, as $n \to \oo$, $\Phi^n$ converges uniformly to $S(t,t_0)\phi$.

\begin{lemma} \label{L:properties}
Let $t_0 \in [0,T]$ and $\phi_1,\phi_2 \in C^2_b(\RR^d)$, and choose $h > 0$ such that $S(t,t_0)\phi_1$ and $S(t,t_0)\phi_2$ belong to $C^2_b(\RR^d)$ for $t \in (t_0-h,t_0+h)$.
	\begin{enumerate}
	\item[(a)] For any $t \in (t_0 - h, t_0 + h)$ and $k \in \RR$, $S(t,t_0)(\phi_1 + k) = S(t,t_0)\phi_1 + k$.\\
	\item[(b)] For any $t \in (t_0 - h,t_0 + h)$, $\sup_{\RR^d} \pars{ S(t,t_0)\phi_1 - S(t,t_0)\phi_2 } \le \sup_{\RR^d} (\phi_1 - \phi_2)$.\\
	\item[(c)] For any $r,s,t \in (t_0 - h, t_0 + h)$, $S(r,s)S(s,t)\phi_1 = S(r,t)\phi_1$.
	\end{enumerate}
\end{lemma}

Property (b) is simply the comparison principle for smooth solutions of \eqref{E:stochHam}. Note that (b) actually holds with equality, because of property (c), which follows from the uniqueness for \eqref{E:stochHam} and the fact that the equation is reversible in the interval $(t_0 - h, t_0 + h)$.

Observe that \eqref{E:stochHam} can be solved for any $\phi \in C^2(\RR^d)$ for which $\nor{D^2\phi}{\oo} < \oo$, even if $\phi$ and $D\phi$ are unbounded. It will not be necessary to exploit this fact in most parts of the paper, and indeed, some arguments below require $\phi$ to be at least Lipschitz.

By estimating the deviation of the characteristic $\mbf X(x,t)$ from its starting point $x$, we obtain the following domain of dependence property for $S(t,t_0)$.

\begin{lemma} \label{L:domainofdependence}
For every $R > 0$, there exists a nondecreasing, continuous function $\rho_R: [0,\oo) \to [0,\oo)$ with $\rho_R(0) = 0$ such that, if $K \subset \RR^d$ is compact, $\phi_1,\phi_2 \in C^2_b(\RR^d)$ satisfy $\nor{D\phi_1}{\oo}, \nor{D\phi_2}{\oo} \le R$, and $h > 0$ is such that $S(t,t_0)\phi_1, S(t,t_0)\phi_2 \in C^2_b(\RR^d)$ for all $t \in (t_0 - h, t_0 + h)$ and $K_{\rho_R(h)}$ is nonempty, then, for all $t \in (t_0 - h, t_0 + h)$,
	\[
		\sup_{K_{\rho_R(|t-t_0|)}} \pars{ S(t,t_0)\phi_1 - S(t,t_0)\phi_2 } \le \sup_{K} (\phi_1 - \phi_2).
	\]
\end{lemma}

Observe that the inequality is vacuous if $K$ has empty interior. In Section \ref{S:steps}, we invoke Lemma \ref{L:domainofdependence} for a closed annulus $K$.

\begin{proof}[Proof of Lemma \ref{L:domainofdependence}]
	Set
	\[
		\rho_R(\sigma) := \sup_{|p| \le R} \sup_{|t-t_0|\le \sigma} \sup_{x \in \RR^d} \abs{ X(x,p,t) - x}.
	\]
	The modulus of continuity of $X$ is uniform for bounded $p$, and otherwise depends only on $\mbf W$ and the derivatives of $H$. Therefore, $\rho_R$ is finite, nondecreasing, continuous, and satisfies $\rho_R(0) = 0$.
	
	For $i = 1, 2$, let $(\mbf X_i, \mbf P_i, \mbf Z_i)$ be as in \eqref{E:solutionchars} for $\phi_i$, and notice that, for any $t \in (t_0 - h, t_0 + h)$,
	\begin{equation}\label{E:charsdeviate}
		\abs{ \mbf X_i\nv(x,t) - x} = \abs{ \mbf X_i\nv(x,t) - X( \mbf X_i\nv(x,t), D\phi_i( \mbf X_i\nv(x,t)), t) }
		\le \rho_R\pars{ |t - t_0|}.
	\end{equation}
	
	Suppose first that $\phi_1 = \phi_2$ in $K$, and let $x$ be in the interior of $K_{\rho_R(|t-t_0|)}$; that is, $\dist(x, K^c) > \rho_R(|t-t_0|)$. In view of \eqref{E:charsdeviate}, $y := \mbf X_1\nv(x,t)$ lies in the interior of $K$. This implies that $\phi_1(y) = \phi_2(y)$ and $D\phi_1(y) = D\phi_2(y)$, so that $(\mbf X_1, \mbf P_1, \mbf Z_1)(y,t) = (\mbf X_2, \mbf P_2, \mbf Z_2)(y,t)$. Therefore $y = \mbf X_2\nv(x,t)$, and
	\[
		S(t,t_0)\phi_1(x) = \mbf Z_1(\mbf X_1\nv(x,t),t) = \mbf Z_1(y,t) = \mbf Z_2(y,t) 
		= \mbf Z_2(\mbf X_2\nv(x,t),t) = S(t,t_0)\phi_2(x).
	\]
	By continuity, the equality is true for any $x \in K_{\rho_R(|t-t_0|)}$.
		
	Now assume $\phi_1 \le \phi_2$ in $K$, fix $\epsilon > 0$, and let $\tilde \phi_2 \in C^2_b(\RR^d)$ be such that $\phi_2 = \tilde \phi_2$ in $K$, $\phi_1 \le \tilde \phi_2 + \epsilon$ in $\RR^d$, and $\nor{D \tilde \phi_2}{\oo} \le R$. Then Lemma \ref{L:properties}(a) yields, for all $x \in K_{\rho_R(|t-t_0|)}$,
	\[
		S(t,t_0)\phi_1(x) \le S(t,t_0)(\tilde \phi_2 + \epsilon)(x) = S(t,t_0)(\phi_2 + \epsilon)(x) = S(t,t_0)\phi_2(x) + \epsilon.
	\]
	Letting $\epsilon \to 0$ finishes the proof in this case. For general $\phi_1$ and $\phi_2$, the result follows from Lemma \ref{L:properties}(a) and the fact that $\phi_1 \le \phi_2 + \sup_{K} (\phi_1 - \phi_2)$ in $K$.
\end{proof}

Note that, when $m = 1$, 
\[
	\rho_R(\sigma) = \sup_{|p| \le R} \sup_{|t - t_0| \le \sigma} \sup_{x \in \RR^d} \abs{ D_p H(p,x)} \abs{ W_t - W_{t_0}},
\]
in accordance with the classical result on finite speed of propagation for Hamilton-Jacobi equations.

\subsection{The definition of pathwise viscosity solutions} The local-in-time spatially-smooth solution operator $S(t,t_0)$ is used to define sub- and super-solutions for the original problem \eqref{E:eq}. In analogy with the classical viscosity solution theory, test functions of the form $S(t,t_0)\phi$ are used to cancel out the ``rough part'' of \eqref{E:eq} (the term involving $dW^i$).

\begin{definition} \label{D:weaksol}
	A function $u \in USC(\RR^d \times [0,T])$ (resp. $u \in LSC(\RR^d \times [0,T])$) is called a pathwise viscosity sub-solution (resp. super-solution) of \eqref{E:eq} if $u$ is bounded from above (resp. from below), $u(\cdot,0) \le u_0$ (resp. $u(\cdot,0) \ge u_0$), and, whenever $\phi \in C^2_b(\RR^d)$, $\psi \in C^1([0,T])$, $h > 0$, $S(t,t_0)\phi \in C^2_b(\RR^d)$ for $t \in (t_0-h,t_0+h)$, and 
\[
	u(x,t) - S(t,t_0)\phi(x) - \psi(t)
\]
attains a local maximum (resp. minimum) at $(x_0,t_0) \in \RR^d \times (t_0-h,t_0+h)$, then
	\begin{equation}\label{E:solutionineq}
	\begin{split}
		&\psi'(t_0) \le F(D^2 \phi(x_0,t_0), D\phi(x_0,t_0), u(x_0,t_0),x_0,t_0)\\[1.2mm]
		&\pars{ \text{resp. } \psi'(t_0) \ge F(D^2 \phi(x_0,t_0), D\phi(x_0,t_0), u(x_0,t_0),x_0,t_0)}.
	\end{split}
	\end{equation} 
	A solution of \eqref{E:eq} is both a sub- and super-solution.
\end{definition}

The following remarks regarding Definition \ref{D:weaksol} are useful in many arguments, and are analogous to observations from the classical viscosity theory.

\begin{lemma} \label{L:weaksol} 
	\begin{enumerate}
	\item[(a)] Assume that $u$ satisfies the hypotheses of Definition \ref{D:weaksol}, except that \eqref{E:solutionineq} only holds when $u(x,t) - S(t,t_0)\phi(x) - \psi(t)$ attains a strict maximum (resp. minimum) at $(x_0,t_0)$, that is, when
	\[
		u(x,t) - S(t,t_0)\phi(x) - \psi(t) \le u(x_0,t_0) - \phi(x_0) - \psi(t_0) \quad \text{(resp. $\ge$)}
	\]
	for all $(x,t) \in \RR^d \times (t_0+h,t_0+h)$, with equality if and only if $(x,t) = (x_0,t_0)$. Then $u$ is a pathwise viscosity sub- (resp. super-) solution in the sense of Definition \ref{D:weaksol}. \\
	
	\item[(b)] If $0 < t_0 \le T$ and $u$ is a sub- (resp. super-) solution in $\RR^d \times (0,t_0)$, then it is a sub- (resp. super-) solution in $\RR^d \times (0, t_0]$. 
	\end{enumerate}
\end{lemma}

It follows that it is sufficient to consider strict maxima or minima, as well as maxima or minima over half open neighborhoods like $B_r(x_0) \times (t_0 - r,t_0]$ instead of $N_r(x_0,t_0)$.

\begin{proof}[Proof of Lemma \ref{L:weaksol}] Since the proofs for sub- and super-solutions are similar, we only present the sub-solution case.

	(a) Assume that $u$ is upper-semicontinuous and bounded from above, and $u(x,t) - S(t,t_0)\phi(x) - \psi(t)$ attains a local maximum at $(x_0,t_0) \in \RR^d \times (t_0 - h,t_0 + h)$. In view of Lemma \ref{L:properties}(a), we may assume, without loss of generality, that $u(x_0,t_0) = \phi(x_0)$. In particular, for some $r > 0$,
	\[
		u(x,t_0) - \phi(x) \le 0 \quad \text{for all $x \in B_r(x_0)$.}
	\]
	Choose $\tilde \phi \in C^2_b(\RR^d)$ such that $\tilde \phi(x) = \phi(x) + |x-x_0|^4$ for $x \in B_r(x_0)$ and $u(\cdot,t_0) < \tilde \phi$ on $\RR^d \backslash B_r(x_0)$, and set $\tilde \psi(t) := \psi(t) + |t - t_0|^2$. Then, in view of Lemma \ref{L:properties}(b), 
	\[
		u(x,t) - S(t,t_0)\tilde \phi(x) - \tilde \psi(t)
	\]
	attains a strict maximum at $(x_0,t_0)$. The result now follows from the fact that $D\tilde \phi(x_0) = D\phi(x_0)$, $D^2 \tilde \phi(x_0) = D^2 \phi(x_0)$, and $\tilde \psi'(t_0) = \psi'(t_0)$.
	
	(b) Assume that $u$ is upper-semicontinuous and bounded from above, and, for some $r >0$, $u(x,t) - S(t,t_0)\phi(x) - \psi(t)$ attains a maximum in $B_r(x_0) \times (t_0 - r, t_0]$ at $(x_0,t_0)$. By replacing $\phi$ and $\psi$ with respectively $\tilde \phi$ and $\tilde \psi$ as in part (a), the maximum may be assumed to be strict over $\RR^d \times (t_0 - r, t_0]$.
	
	Fix $\nu > 0$ and assume that $(x_\nu, t_\nu)$ is a maximum point for 
	\[
		u(x,t) - S(t,t_0)\phi(x) - \psi(t) - \frac{\nu}{t_0 - t}
	\]
	over $\oline{B_1(x_0)} \times [t_0 - r, t_0]$. Then $t_\nu \in [t_0 - r, t_0)$ for all $\nu > 0$, because $u$ is bounded. Let $(y,s) \in \oline{B_1(x_0)} \times [t_0 - r, t_0]$ be an accumulation point of the sequence $\{(x_\nu,t_\nu)\}_{\nu > 0}$ as $\nu \to 0$, and assume that $s \ne t_0$. For fixed $(x,t) \in \oline{B_1(x_0)} \times [t_0 - r, t_0)$,
	\[
		u(x,t) - S(t,t_0)\phi(x) - \psi(t) - \frac{\nu}{t_0 - t} \le u(x_\nu, t_\nu) - S(t_\nu,t_0)\phi(x_\nu) - \psi(t_\nu) - \frac{\nu}{t_0 - t_\nu}.
	\]
	Letting $\nu \to 0$ along a subsequence such that $(x_\nu,t_\nu) \to (y,s)$ yields
	\[
		u(x,t) - S(t,t_0)\phi(x) - \psi(t) \le u(y,s) - S(s,t_0)\phi(y) - \psi(s),
	\]
	and, in view of the semicontinuity of $u$, the same inequality holds for $(x,t) = (x_0,t_0)$, contradicting the strictness of the maximum point $(x_0,t_0)$. It follows that the whole sequence $(x_\nu,t_\nu)$ converges to $(x_0,t_0)$, and in particular, for sufficiently small $\nu$, $(x_\nu,t_\nu) \in B_1(x_0) \times (t_0 - r, t_0)$. Therefore, Definition \ref{D:weaksol} yields
	\[
		\psi'(t_\nu) \le \psi'(t_\nu) + \frac{\nu}{(t_0 - t_\nu)^2} \le F(D^2 S(t_\nu,t_0)\phi(x_\nu), D S(t_\nu,t_0)\phi(x_\nu), u(x_\nu,t_\nu), x_\nu, t_\nu),
	\]
	and the proof is finished upon letting $\nu \to 0$.
\end{proof}

As is well known, the method of characteristics cannot be used to solve general nonlinear second-order equations like \eqref{E:eq} even locally in time. However, using the solution operators $S(t,t_0)$, it is possible to construct global sub- and super-solutions of \eqref{E:eq} with smooth initial data.

\begin{lemma} \label{L:subsupers}
	For every $\phi \in C^2_b(\RR^d)$, there exist a locally bounded sub- and super-solution $\uline{u}$ and $\oline{u}$ of \eqref{E:eq} with $\uline{u}(\cdot,0) = \oline{u}(\cdot,0) = \phi$. Moreover, for some $h > 0$, $\uline{u}$ and $\oline{u}$ are continuous on $\RR^d \times [0,h]$.
\end{lemma}

\begin{proof}
Only the sub-solution is constructed, since the argument for the super-solution is similar.

For some $h > 0$, there exists a solution $\Phi(x,t) = S(t,0)\phi(x)$ of \eqref{E:stochHam} satisfying $\Phi \in C([0,h], C^2_b(\RR^d))$. Set
\[
	R := \sup_{0 \le t \le h} \nor{\Phi(\cdot,t)}{C^2_b(\RR^d)}, \quad
	C := \inf_{|X| + |p| + |u| \le R, \; (x,t) \in \RR^d \times [0,T]} F(X,p,u,x,t),
\]
and $\uline{u}(\cdot,t) := \Phi(\cdot,t) - Ct$ for $t \in [0,h]$.

If $\Phi_0$ is defined by $\Phi_0(x,t,s) := S(t,s)(0)(x)$, then, for some $h_0 > 0$ that is independent of $s$, $\Phi_0(\cdot,t,s) \in C^2_b(\RR^d)$ when $s \le t \le s + h_0$. Note that Lemma \ref{L:properties}(a) implies that $S(t,s)(M)(x) = \Phi_0(x,t,s) + M$ for all $M \in \RR$. 

Define
\[
	R_0 := \sup_{0 \le t-s \le h_0} \nor{\Phi_0(\cdot,t,s)}{C^2_b(\RR^d)}, \quad
	C_0 := \inf_{|X| + |p| + |u| \le R_0, (x,t) \in \RR^d \times [0,T]} F(X,p,u,x,t), \quad
	M_0 := \sup_{x \in \RR^d} \uline{u}(x,h)_-,
\]
and, for $k = 0,1,2,3,\ldots, \left\lceil \frac{T-h}{h_0} \right\rceil - 1$ and $t \in (h + k h_0, h + (k+1)h_0]$, 
\[	
	M_k := \sup_{x \in \RR^d} \uline{u}(x, h + kh_0)_-  \quad \text{and} \quad
	\uline{u}(\cdot,t) := \Phi_0(\cdot,t,h + kh_0) - M_k - C_0(t - h - k h_0). 
\]
By construction, $\uline{u}$ is upper-semicontinuous, bounded from above, and continuous on $\RR^d \times [0,h]$.

Now choose $\eta \in C^2_b(\RR^d)$ and $\psi \in C^1([0,T])$, and assume, for some $h_1 > 0$ and $(x_0,t_0) \in \RR^d \times (0,T]$, that $S(t,t_0)\eta \in C^2_b(\RR^d)$ for $t \in (t_0 - h_1, t_0 + h_1)$ and 
\[
	\uline{u}(x,t) - S(t,t_0)\eta(x) - \psi(t)
\]
attains a strict maximum in $\RR^d \times (t_0 - h_1, t_0 + h_1)$ at $(x_0,t_0)$. In view of Lemma \ref{L:weaksol}(b), it suffices to consider $t_0 \ne h + kh_0$ for any $k \in \NN \cup \{0\}$. Assume also that $t_0 \in (h + kh_0, h + (k+1)h_0)$ for some $k \ge 0$, as the proof for $t_0 \in (0,h)$ is similar. Then 
\[
	D\Phi_0(x_0,t_0) = D\eta(x_0), \quad D^2\Phi_0(x_0,t_0) \le D^2 \eta(x_0), \quad \text{and} \quad \psi'(t_0) = -C_0,
\]
where the last equality follows from Lemma \ref{L:properties}(c). Therefore,
\begin{gather*}
	\psi'(t_0) - F(D^2\eta(x_0), D\eta(x_0), \uline{u}(x_0,t_0), x_0, t_0)\\[1.2mm]
	\le -C_0 - F(D^2 \Phi_0(x_0,t_0), D\Phi_0(x_0,t_0), \Phi_0(x_0,t_0, h + kh_0), x_0, t_0) \le 0,
\end{gather*}
and so $\uline{u}$ is a sub-solution of \eqref{E:eq}.
\end{proof}

\section{The Perron construction} \label{S:steps}

The proof of Theorem \ref{T:perrons} involves two main steps. First, it is clear from Definition \ref{D:weaksol} that the maximum of a finite number of sub-solutions is also a sub-solution, with a corresponding statement holding true for the minimum of a finite number of super-solutions. We generalize this observation to infinite families with the next result.

\begin{lemma} \label{L:stability}
	Let $\mcl F$ be a family of sub- (resp. super-) solutions of \eqref{E:eq}. Define
	\[
		U(x,t) := \sup_{ v \in \mcl F} v(x,t) \quad \pars{\text{ resp. } \inf_{v \in \mcl F} v(x,t) } .
	\]
	Assume that $U^*<\oo$ (resp. $U_*> -\oo$). Then $U^*$ (resp. $U_*$) is a sub- (resp. super-) solution of \eqref{E:eq}.
\end{lemma}

The second step is to show that if a ``strict'' sub-solution has its values increased in a certain way in a sufficiently small neighborhood, then the resulting function is another sub-solution. This ``bump'' construction is less straightforward than in the classical viscosity setting, due to the limited flexibility in the choice of test functions, and the domain of dependence result Lemma \ref{L:domainofdependence} plays an important role in the proof.

\begin{lemma} \label{L:bump}
Suppose that $w$ is a sub-solution of \eqref{E:eq}, and that $w_*$ fails to be a super-solution. Then, for some $(x_0,t_0) \in \RR^d \times (0,T]$ and for all $\kappa > 0$, there exists a sub-solution $w_\kappa$ of \eqref{E:eq} such that
	\begin{align*}
		w_\kappa \ge w, \quad
		\sup(w_\kappa - w) > 0, \quad \text{and} \quad
		w_\kappa = w \quad \text{in} \quad \pars{ \RR^d \times [0,T]} \backslash N_{\kappa} (x_0,t_0).
	\end{align*}
\end{lemma}

We present the proof of Theorem \ref{T:perrons} before those of Lemmas \ref{L:stability} and \ref{L:bump}, to emphasize their importance in the argument. Recall that $u$ is defined via the formula \eqref{E:solution}.

\begin{proof}[Proof of Theorem \ref{T:perrons}]
	Observe first that, in view of Lemma \ref{L:subsupers} and the comparison principle \eqref{A:comparison}, $u$ is well-defined and bounded. 
	
	Fix $\epsilon > 0$, let $\phi^\epsilon \in C^2_b(\RR^d)$ be such that 
	\[
		\phi^\epsilon - \epsilon \le u_0 \le \phi^\epsilon + \epsilon \quad \text{on } \RR^d,
	\] 
	let $\uline{u}^\epsilon$ and $\oline{u}^\epsilon$ be the sub- and super-solution given by Lemma \ref{L:subsupers} corresponding respectively to the initial conditions $\phi^\epsilon - \epsilon$ and $\phi^\epsilon + \epsilon$, and fix $h > 0$ such that both $\uline{u}^\epsilon$ and $\oline{u}^\epsilon$ are continuous on $\RR^d \times [0,h]$. Then the comparison principle yields
	\[
		\uline{u}^\epsilon \le u_* \le u \le u^* \le \oline{u}^\epsilon \quad \text{on } \RR^d \times [0,h],
	\]
	and therefore $\phi^\epsilon - \epsilon \le u_*(\cdot, 0) \le u^*(\cdot, 0) \le \phi^\epsilon + \epsilon$. Since $\epsilon$ is arbitrary, it follows that $u(\cdot, 0) = u_0$ and $u$ is continuous on $\RR^d \times \{0\}$.

	In view of Lemma \ref{L:stability}, $u^*$ is a sub-solution of \eqref{E:eq}. Formula \eqref{E:solution} implies that $u^* \le u$, and therefore $u^* = u$. That is, $u$ is itself upper-semicontinuous and a sub-solution.
	
	On the other hand, $u_*$ is a super-solution. If this were not the case, then Lemma \ref{L:bump} would imply the existence of a sub-solution $\tilde u \ge u$ and a neighborhood $N \subset \RR^d \times (0,T]$ such that $\tilde u = u$ in $(\RR^d \times [0,T]) \backslash N$ and $\sup_{N} (\tilde u - u) > 0$, contradicting the maximality of $u$.
	
	The comparison principle gives $u^* \le u_*$, and, as a consequence of the definition of semicontinuous envelopes, $u_* \le u^*$. Therefore, $u = u_* = u^*$ is a solution of \eqref{E:eq} with $u = u_0$ on $\RR^d \times \{0\}$. The uniqueness of $u$ follows from yet another application of the comparison principle.
\end{proof}

\begin{proof}[Proof of Lemma \ref{L:stability}]
We give only the proof for sub-solutions, since it is almost identical for super-solutions.
	
Let $\phi \in C^2_b(\RR^d)$, $\psi \in C^1([0,T])$, $t_0 > 0$, and $h > 0$ be such that $S(t,t_0)\phi \in C^2_b(\RR^d)$ for all $t \in (t_0 - h, t_0 + h)$, assume that
\[
	U^*(x,t) - S(t,t_0)\phi(x) - \psi(t)
\]
attains a local maximum at $(x_0,t_0) \in \RR^d \times (t_0 - h, t_0 + h)$, and, without loss of generality, assume $x_0 = 0$, $\phi(0) = 0$, and $\psi(t_0) = 0$. Set $p := D\phi(0)$, $X := D^2\phi(0)$, and $a := \psi'(t_0)$. 
	
For fixed $\delta > 0$, let $r > 0$ be such that
\[
	\phi(x) \le \ip{p,x} + \frac{1}{2} \ip{Xx,x} + \delta |x|^2 \quad \text{and} \quad \psi(t) \le a(t - t_0) + \delta |t - t_0| \quad \text{for all} \quad (x,t) \in N_{r}(0,t_0).
\]
Choose $\phi_1$ and $\phi_2$ in $C^2_b(\RR^d)$ such that
\[
	\begin{dcases}
		\phi_1(x) = \ip{p,x} + \frac{1}{2} \ip{Xx,x} + \delta |x|^2  & \text{for }  x \in B_{r}(x_0),\\[1.2mm]
		\phi_2(x) = \ip{p,x} + \frac{1}{2} \ip{Xx,x} + 2 \delta |x|^2 & \text{for }  x \in B_{r}(x_0),  \text{ and} \\[1.2mm]
		\phi \le \phi_1 \le \phi_2 & \text{on } \RR^d.
	\end{dcases}
\]
Shrinking $h$ if necessary, assume that, for $t \in (t_0 - h, t_0 + h)$, $S(t,t_0)\phi_1$ and $S(t,t_0)\phi_2$ belong to $C^2_b(\RR^d)$. Then Lemma \ref{L:properties}(b) implies that, for all $t \in (t_0 - h, t_0 + h)$, $S(t,t_0)\phi \le S(t,t_0) \phi_1 \le S(t,t_0) \phi_2$, and, in particular, $U^*(x,t) - S(t,t_0) \phi_1(x) - \psi(t)$ attains a local maximum at $(0,t_0)$.
	
Let $(x_n,t_n) \in \RR^d \times (t_0 - h, t_0 + h)$ and $v_n \in \mcl F$ be such that, as $n \to \oo$, $(x_n, t_n) \to (0, t_0)$ and $v_n(x_n,t_n) \to U^*(0,t_0)$, and let $(x_n',t_n')$ be the maximum point attained over $\oline{N_r(0,t_0)}$ by the function
\begin{align}
	v_n(x,t) - S(t,t_0) \phi_2(x) - a(t - t_0) - 2\delta (|t- t_0|^2 + n\nv)^{1/2}. \label{E:max}
\end{align} 
Then
\begin{equation}\label{E:approxmax}
	\begin{split}
	v_n(x_n,t_n) &\le v_n(x_n',t_n') + S(t_n,t_0) \phi_2(x_n) - S(t_n',t_0)\phi_2(x_n') + a(t_n - t_n')\\[1.2mm]
	& + 2\delta \pars{ (|t_n - t_0|^2 + n\nv)^{1/2} - (|t_n' - t_0|^2 + n\nv)^{1/2}}.
	\end{split} 
\end{equation}
Let $(y,s) \in \bbar{N_r(0,t_0)}$ be an accumulation point of the sequence $\{ (x_n',t_n') \}_{n \in \NN}$. Passing to the limit in \eqref{E:approxmax} yields
\begin{equation}\label{E:forcetopoint}
	\begin{split}
	U^*(0,t_0) &\le U^*(y,s) - S(s,t_0) \phi_2(y)- a(s - t_0) - 2\delta |s - t_0|  \\[1.2mm]
	&\le U^*(0,t_0)  + S(s,t_0) \phi_1 (y) - S(s,t_0)\phi_2(y) - \delta |s- t_0| \le U^*(0,t_0) - \delta |s-t_0|,
	\end{split}
\end{equation}
and, therefore, $s = t_0$. Inserting this fact into \eqref{E:forcetopoint} gives $\phi_2(y) \le \phi_1(y)$, which implies that
\[
	y = 0, \quad \lim_{n \to \oo} (x_n',t_n') = (0,t_0), \quad \text{and} \quad \lim_{n \to \oo} v_n(x_n',t_n') = U^*(0,t_0).
\]
In particular, for sufficiently large $n$, $(x_n',t_n') \in N_r(0,t_0)$.
	 
Finally, set $\Phi(x,t) := S(t,t_0) \phi_2(x)$. Definition \ref{D:weaksol} gives
\[
	a + 2 \delta \frac{ t_n' - t_0}{ (|t_n' - t_0|^2 + n\nv)\nv} \le F(D^2 \Phi(x_n',t_n'), D\Phi(x_n',t_n'), v_n(x_n',t_n'), x_n', t_n').
\]
Upon letting $n \to \oo$ and $\delta \to 0$, this becomes $a \le F(X,p, U^*(0,t_0),0,t_0)$, as desired.
\end{proof}

\begin{proof}[Proof of Lemma \ref{L:bump}]
By assumption, there exist $\phi \in C^2_b(\RR^d)$, $\psi \in C^1([0,T])$, $(x_0,t_0) \in \RR^d \times (0,T]$, and $h \in (0,\kappa)$ such that $S(t,t_0)\phi \in C^2_b(\RR^d)$ for $t \in (t_0 -h, t_0 + h)$, 
\[
	w_*(x,t) - S(t,t_0)\phi(x) - \psi(t)
\]
attains a local minimum at $(x_0,t_0)$, and
\begin{equation}
	\psi'(t_0) - F(D^2\phi(x_0),D\phi(x_0), w_*(x_0,t_0),x_0,t_0) < 0. \label{E:inequality}
\end{equation}
Assume again $x_0 = 0$, $\phi(0) = 0$, and $\psi(t_0) = 0$, set $X := D^2\phi(0)$, $p := D\phi(0)$, and $a := \psi'(t_0)$, and define the nondecreasing functions $\omega_1, \omega_2: [0,\oo) \to [0,\oo)$ by
\[
	\omega_1(\sigma) := \sup_{|x| \le \sigma} \frac{|\phi(x) - \ip{p,x} - \half \ip{Xx,x}|}{|x|^2} \quad \text{and} \quad
	\omega_2(\sigma) := \sup_{|t - t_0| \le \sigma} \frac{|\psi(t) - a(t-t_0)|}{|t-t_0|}.
\]

Let $\gamma \in (0,1)$, $r \in (0,\kappa)$, and $s \in (0,h)$ be such that
\begin{equation}
	\omega_1(r), \omega_2(s) \le \frac{\gamma}{2}, \label{E:rs}
\end{equation}
and set
\begin{equation}
	\delta := \gamma \min\pars{ \frac{ r^2}{16}, \frac{ s}{8}}. \label{E:delta}
\end{equation}
Choose $\hat \eta \in C_b^2(\RR^d)$ so that
\[
	\hat \eta(x) = \ip{p,x} + \half \ip{Xx,x} - \gamma|x|^2 \text{ in } B_r(x_0) \quad \text{and} \quad \hat \eta \le \phi \text{ in } \RR^d.
\]
Redefining $h > 0$ to be smaller, if necessary, we may assume that $S(t,t_0)\hat \eta \in C^2_b(\RR^d)$ for $t \in (t_0 - h, t_0 + h)$. Observe that this may also result in $s$, and therefore $\delta$, becoming smaller.

For $(x,t) \in \RR^d \times (t_0 - h, t_0 + h)$, define
\[
	\hat w(x,t) := w_*(0,t_0) + \delta + S(t,t_0)\hat \eta(x) + a(t-t_0) - \gamma(|t - t_0|^2 + \delta^2)^{1/2}.
\]
Then, if $\gamma$, $r$, and $s$ (and therefore $\delta$) are sufficiently small, $\hat w$ satisfies the sub-solution property in $N_{r,s}(0,t_0)$. Indeed, assume, for some $\zeta \in C^2_b(\RR^d)$, $\alpha \in C^1([0,T])$, $\hat h > 0$, and $(\hat x, \hat t) \in N_{r,s}(0,t_0)$, that $S(t,t_0)\zeta \in C^2_b(\RR^d)$ for $t \in (\hat t - \hat h, \hat t + \hat h)$ and 
\[
	\hat w(x,t) - S(t,\hat t) \zeta(x) - \alpha(t)
\]
attains a strict maximum at $(\hat x, \hat t)$. This implies that 
\[
	D S(\hat t, t_0) \hat \eta(\hat x) = D\zeta(\hat x) \quad \text{and} \quad D^2 S(\hat t,t_0) \hat \eta(\hat x) \le D^2 \zeta(\hat x),
\]
and, because 
\[
	t \mapsto \sup_{\RR^n} \pars{ S(t,t_0) \hat \eta - S(t,\hat t) \zeta} + at - \gamma(|t-t_0|^2 + \delta^2)^{1/2} - \alpha(t)
\]
attains a maximum at $\hat t$, Lemma \ref{L:properties}(c) yields
\[
	a - \gamma \frac{\hat t - t_0}{\pars{|\hat t - t_0|^2 + \delta^2}^{1/2}} = \alpha'(\hat t).
\]
Therefore, in view of the strict inequality in \eqref{E:inequality}, the continuity of the solution map $S(t,t_0)$ on $C^2_b(\RR^d)$, and the continuity of $F$, it follows that $\alpha'(\hat t) \le F(D^2 \zeta(\hat x), D\zeta(\hat x),  \hat w(\hat x, \hat t), \hat x, \hat t)$ if $\gamma$, $r$, and $s$ are small enough.

Define $R := \max_{|t-t_0| \le h} \max \left\{ \nor{DS(t,t_0)\phi}{\oo}, \nor{DS(t,t_0)\hat \eta }{\oo} \right\}$, and shrink $s$ further so that
\begin{equation}\label{E:pathrestrict}
	\rho_R(s) \le \frac{r}{8},
\end{equation}
where $\rho_R$ is the modulus from Lemma \ref{L:domainofdependence}. The claim is that
\begin{equation}
	w(x,t) > \hat w(x,t) \quad \text{in } N_{7r/8,s}(0,t_0) \backslash \oline{N_{5r/8, s/2}(0,t_0)}. \label{E:bump}
\end{equation}
As a first step, observe that 
\[
	w(x,t) - \hat w(x,t) \ge w_*(x,t) - \hat w(x,t) \ge -\delta + S(t,t_0)\phi(x) - S(t,t_0)\hat \eta(x) 
	+ (\gamma- \omega_2(s)) |t-t_0|.
\]
	
Suppose that $(x,t) \in N_{7r/8,s}(0,t_0) \backslash \oline{N_{7r/8,s/2}}$, that is, $s/2 < |t-t_0| < s$ and $|x| < 7r/8$. Then \eqref{E:rs} and \eqref{E:delta} give
\begin{equation}
	w(x,t) - \hat w(x,t)
	\ge -\delta + (\gamma - \omega_2(s)) \cdot \frac{s}{2}
	\ge -\delta + \frac{\gamma s}{4}
	\ge \frac{\gamma s}{8} > 0.
\end{equation}
To prove \eqref{E:bump} for $(x,t) \in N_{7r/8,s} \backslash \oline{N_{5r/8,s}}$, we apply Lemma \ref{L:domainofdependence} to the annulus $K = \bbar{B_r(0)} \backslash B_{r/2}(0)$ and obtain
\begin{align*}
	\inf_{r/2 \le |x| \le r}  (\phi(x) - \hat \eta(x)) &\le \inf \left\{ S(t,t_0)\phi(x) - S(t,t_0)\hat \eta(x) : \dist(x, K^c) \ge \rho_R(|t-t_0|)  \right\}\\[1.2mm]
	 &\le \inf \left\{ S(t,t_0)\phi(x) - S(t,t_0)\hat \eta(x) : \min\pars{ r - |x|, |x| - \frac{r}{2} } \ge \rho_R(s) \right\}.
\end{align*}
Combining this with \eqref{E:rs}, \eqref{E:delta}, and \eqref{E:pathrestrict}, it follows that, whenever $5r/8 < |x| < 7r/8$ and $|t-t_0| < s$,
\begin{align*}
	w(x,t) - \hat w(x,t)
	&\ge -\delta + \inf_{r/2 \le |x| \le r}(\phi(x) - \hat \eta(x)) 
	\ge -\delta + (\gamma - \omega_1(r)) \cdot \pars{\frac{r}{2}}^2 \\[1.2mm]
	&\ge -\delta + \frac{\gamma r^2}{8} 
	\ge \frac{\gamma r^2}{16} > 0.
\end{align*}
This finishes the proof of \eqref{E:bump}.
	
Finally, define
\[
	w_{\kappa}(x,t) := 
	\begin{cases}
		\max( \hat{w}(x,t), w(x,t) ) & \text{for } (x,t) \in N_{7r/8,s}(0,t_0), \text{ and} \\[1.2mm]
		w(x,t) & \text{for } (x,t) \notin N_{7r/8,s}(0,t_0).
	\end{cases}
\]
Then $w_\kappa \ge w$, and $w_\kappa = w$ outside of $N_{\kappa}(0,t_0)$. If $(x_n,t_n)$ is such that $\lim_{n \to \oo} (x_n,t_n) = (0,t_0)$ and $\lim_{n \to \oo} w(x_n,t_n) = w_*(0,t_0)$, then 
\[
	\lim_{n \to \oo} \pars{ w(x_n,t_n) - \hat w(x_n,t_n) }= -(1 - \gamma)\delta < 0,
\]
so that $\sup_{N_{\kappa}(0,t_0)} \pars{ w_\kappa - w} > 0$. Finally, $w_\kappa$ is a sub-solution. This is evident on $(\RR^d \times [0,T]) \backslash \oline{N_{7r/8,s}(0,t_0)}$, as well as in the interior of $N_{7r/8,s}(0,t_0)$, because there $w_\kappa$ is equal to the pointwise maximum of two sub-solutions. It remains to verify the sub-solution property on the boundary of $N_{7r/8,s}(0,t_0)$, and this follows because, in view of \eqref{E:bump}, $w_\kappa = w$ in a neighborhood of the boundary of $N_{7r/8,s}(0,t_0)$.
\end{proof}

\appendix
\section{The comparison principle}

The following is a discussion of the various situations in which a comparison principle for \eqref{E:eq} is known. It is always assumed below that $H$ satisfies \eqref{A:Hregularity}, and, for second order equations, $F$ is assumed to satisfy \eqref{A:Fassumptions} and \eqref{A:Fstructure}. The object of this section is not to prove the comparison principle, but rather to outline the general strategy, which has been adapted to several situations in various published or forthcoming works.

\subsection{The general strategy and some preliminary observations}
As is usual in the theory of viscosity solutions, the comparison principle is proved by doubling the space variable and using a smooth test function to penalize away from the diagonal. Note first that it suffices to prove the comparison principle when the sub-solution $u$ and super-solution $v$ satisfy $u(x,0) \le v(x,0)$ on $\RR^d$, since otherwise, a positive constant can be added to $v$, exploiting the monotonicity of $F$ in the $v$ variable. Then, if the comparison principle fails, there exist $\mu > 0$ and $t > 0$ such that
\[
	\sup_{x \in \RR^d} \left\{ u(x,t) - v(x,t) \right\}_+ > \mu t.
\]
To avoid technical details, assume that $u$ and $v$ are both periodic on $\RR^d$. Then, for sufficiently large $\lambda > 0$,
\begin{equation} \label{E:doubledquantity}
	u(x,t) - v(y,t) - \frac{\lambda}{2} |x-y|^2 - \mu t
\end{equation}
attains a maximum at $(x_\lambda, y_\lambda, t_\lambda) \in \RR^d \times \RR^d \times (0,T]$. If $u$ and $v$ are not periodic, then the maximum of the above function may not actually be achieved, and further penalizations are required to make the argument work.

The effect of doubling variables can be seen immediately in the first-order case. 
\begin{lemma} \label{L:doubledeq}
Assume that $F \equiv 0$, and that $u$ and $v$ are respectively a sub- and super-solution of \eqref{E:eq}. Then the function
\[
	z(x,y,t) := u(x,t) - v(y,t)
\]
is a sub-solution of the equation
\begin{equation}\label{E:doubledeq}
	dz = \sum_{i=1}^m \pars{ H^i(D_x z, x) - H^i(-D_y z, y) } \cdot dW^i.
\end{equation}
\end{lemma}

\begin{proof}
Let $h> 0$, $\psi \in C^1((t_0 - h, t_0 +h))$, and assume that $\Phi \in C((t_0 - h, t_0 +h), C^2_b(\RR^d \times \RR^d))$ is a solution of \eqref{E:doubledeq} and that
\[
	z(x,y,t) - \Phi(x,y,t) - \psi(t)
\]
attains a strict maximum at $(x_0,y_0, t_0)$. As functions on $\RR^{2d} \times \RR^{2d}$, the Hamiltonians $(p,q,x,y) \mapsto H(p,x)$ and $(p,q,x,y) \mapsto -H(-q,y)$ have a Poisson bracket equal to zero (see \eqref{A:Poisson}), so that their characteristic flows commute. In particular, it follows that $\Phi$ is given by the formula
\[
	\Phi(x,y,t) = S^+(t,t_0)S^-(t,t_0)\Phi(\cdot,t_0)(x,y) = S^-(t,t_0)S^+(t,t_0)\Phi(\cdot,t_0)(x,y),
\]
where $S^+$ and $S^-$ are the solution operators for respectively
\[
	dU = \sum_{i=1}^m H^i(D_xU,x) \cdot dW^i \quad \text{and} \quad dU = - \sum_{i=1}^m H^i(-D_y U, y) \cdot dW^i.
\]
That is, $U(x,t) := S^\pm(t,t_0)\phi$ solves the corresponding equations above with $U(\cdot,t_0) = \phi$.

Fix $\delta >0$ and let $(x_\delta,y_\delta, s_\delta, t_\delta)$ be a maximum point of
\[
	u(x,s) - v(y,t) - S^+(s,t_0)S^-(t,t_0)\Phi(\cdot,t_0)(x,y) - \psi(s) - \frac{|s-t|^2}{\delta}
\]
in $\oline{ B_1(x_0) \times B_1(y_0) \times [t_0 - h, t_0 + h]^2}$. Because the original maximum was strict, it follows that $\lim_{\delta \to 0} (x_\delta, y_\delta, s_\delta, t_\delta) = (x_0, y_0, t_0, t_0)$, so that, for sufficiently small $\delta$, $(x_\delta,y_\delta, s_\delta, t_\delta) \in B_1(x_0) \times B_1(y_0) \times (t_0 - h, t_0 + h)$.

Definition \ref{D:weaksol} now yields
\[
	\psi'(s_\delta) + \frac{s_\delta - t_\delta}{\delta} \le 0 \quad \text{and} \quad \frac{s_\delta - t_\delta}{\delta}\ge 0,
\]
so that $\phi'(s_\delta) \le 0$, and therefore $\psi'(t_0) \le 0$ follows after taking $\delta \to 0$.
\end{proof}

In the second-order case, Lemma \ref{L:doubledeq} no longer holds, and it is necessary to prove an analogue of the ``Theorem of Sums'' from the classical viscosity theory (see Theorem 8.3 in \cite{CIL}). 

\begin{lemma}\label{L:sums}
	Assume that $u$ and $v$ are respectively a sub- and super-solution of \eqref{E:eq}, $(x_0,y_0,t_0) \in \RR^d \times \RR^d \times (0,T]$, $h \in (0,t_0)$, $\psi \in C^1((t_0 -h, t_0 +h))$, $\Phi \in C((t_0-h,t_0+h), C^2_b(\RR^d \times \RR^d))$ is a solution of
	\begin{equation} \label{E:smoothdoubledeq}
		d\Phi = \sum_{i=1}^m \pars{ H^i(D_x \Phi,x) - H^i(-D_y \Phi,y)} \cdot dW^i \quad \text{in } \RR^d \times \RR^d \times (t_0 - h, t_0 +h),
	\end{equation}
	and
	\[
		u(x,t) - v(y,t) - \Phi(x,y,t) - \psi(t)
	\]
	attains a local maximum at $(x_0,y_0,t_0) \in \RR^d \times \RR^d \times (t_0 - h, t_0 +h)$. Then, for every $\epsilon > 0$, there exist $X_\epsilon, Y_\epsilon \in S^d$ such that
	\[
		- \pars{ \abs{ D^2 \Phi(x_0,y_0,t_0)} + \frac{1}{\epsilon}} 
		\begin{pmatrix}
			I_d & 0 \\
			0 & I_d
		\end{pmatrix}
		\le
		\begin{pmatrix}
			X_\epsilon & 0 \\
			0 & -Y_\epsilon
		\end{pmatrix}
		\le
		D^2\Phi(x_0,y_0,t_0) + \epsilon \pars{ D^2 \Phi(x_0,y_0,t_0) }^2
	\]
	and
	\[
		\psi'(t_0) \le F(X_\epsilon, D_x \Phi(x_0,y_0,t_0), u(x_0,t_0), x_0, t_0) - F(Y_\epsilon, -D_y \Phi(x_0,y_0,t_0),  v(y_0,t_0), y_0, t_0).
	\]
\end{lemma}

The proof of this lemma appears in \cite{LS3} (see also \cite{Snotes}) in the case where the $H^i$ are independent of space and $\Phi(x,y,t) = (\lambda/2)|x-y|^2$. The general result above follows after making some technical adaptations in the argument. As in the classical viscosity theory, the strategy is to regularize $u$ and $v$ using the so-called inf- and sup-convolutions. The task here is made more complicated by the limited flexibility of the class of test functions, and the regularizations need to be modified accordingly.

In the next sub-sections, we give some indication as to how the above strategy is applied to various situations, depending on the nature of the spatial dependence of $H$, and whether or not $H$ and $W$ are scalar.

\subsection{$H$ independent of $x$} The comparison principle holds without any other restrictions on $F$, $H$, and $W$. Indeed, this can be seen immediately in the first-order case, since the function
\begin{equation} \label{E:stationarydist}
	\Phi_\lambda(x,y,t) := \frac{\lambda}{2} |x-y|^2
\end{equation}
is a stationary smooth solution of \eqref{E:smoothdoubledeq}. Therefore, if $(x_\lambda, y_\lambda, t_\lambda) \in \RR^d \times \RR^d \times (0,T]$ maximizes \eqref{E:doubledquantity}, then Lemma \ref{L:doubledeq} yields $\mu \le 0$, which is a contradiction.

Equipped with Lemma \ref{L:sums}, the proof of the comparison principle in the second order setting proceeds as follows. As explained in the previous sub-section, the argument is reduced to deriving a contradiction out of the fact that $u(\cdot,0) \le v(\cdot,0) $ and \eqref{E:doubledquantity} attains a maximum at $(x_\lambda, y_\lambda, t_\lambda) \in \RR^d \times \RR^d \times (0,T]$ for sufficiently large $\lambda$. Classical arguments give $\lim_{\lambda \to \oo} \frac{\lambda}{2} |x_\lambda - y_\lambda|^2 = 0$ and $\lim_{\lambda \to \oo} (x_\lambda, y_\lambda) = (x_0,x_0)$, where $x_0$ is some point for which $u(x_0,t_0) > v(x_0,t_0) + \mu t_0$, so that, for sufficiently large $\lambda$, $u(x_\lambda, t_\lambda) \ge v(y_\lambda, t_\lambda)$.
	
Lemma \ref{L:sums}, when applied to $\epsilon = \lambda\nv$ and the stationary, smooth solution $\Phi_\lambda$ from \eqref{E:stationarydist}, yields matrices $X_\lambda, Y_\lambda \in S^d$ for which
\[
	- 3\lambda
	\begin{pmatrix}
		I_d & 0 \\
		0 & I_d
	\end{pmatrix}
	\le
	\begin{pmatrix}
		X_\lambda & 0 \\
		0 & -Y_\lambda
	\end{pmatrix}
	\le
	3\lambda
	\begin{pmatrix}
		I_d & -I_d \\
		-I_d & I_d
	\end{pmatrix}
\]
and
\begin{align*}
	\mu &\le F(X_\lambda, \lambda(x_\lambda - y_\lambda), u(x_\lambda,t_\lambda), x_\lambda, t_\lambda) - F(Y_\lambda, \lambda(x_\lambda - y_\lambda), v(y_\lambda,t_\lambda), y_\lambda, t_\lambda)\\
	&\le F(X_\lambda, \lambda(x_\lambda - y_\lambda), u(x_\lambda,t_\lambda), x_\lambda, t_\lambda) - F(Y_\lambda, \lambda(x_\lambda - y_\lambda), u(x_\lambda,t_\lambda), y_\lambda, t_\lambda).
\end{align*}
Therefore, assumption \eqref{A:Fstructure} with $C = 3$ gives
\[
	\mu \le \omega \pars{ \lambda|x_\lambda - y_\lambda|^2 + |x_\lambda - y_\lambda|},
\]
and a contradiction is reached upon taking $\lambda$ sufficiently large.

There is a theory, proposed in \cite{LS2}, for pathwise equations dealing with $H$ independent of $x$, but not necessarily smooth (see \cite{Seschemes} or \cite{Snotes} for details). The results obtained in the present paper are not applicable to such equations, since all of the constructions rely on the existence of local-in-time smooth solutions to the rough Hamilton-Jacobi equations.

\subsection{$H$ and $W$ are scalar, and $H$ depends on $x$} 
The dependence of $H$ on $x$ immediately complicates the above method of proof, because \eqref{E:stationarydist} is no longer an exact solution of \eqref{E:smoothdoubledeq}. 

One approach is to adapt the strategy from the classical viscosity case, and consider Hamiltonians for which $\frac{\lambda}{2} |x-y|^2$ is ``almost'' a solution of \eqref{E:smoothdoubledeq}. In the pathwise setting, this entails the study of the equation
\begin{equation} \label{E:specificdoubledeq}
	d\Phi_\lambda = \pars{ H(D_x \Phi_\lambda, x) - H(-D_y \Phi_\lambda,y) } \cdot dW^i \quad \text{in } \RR^d \times \RR^d \times (t_0 - h_\lambda, t_0 + h_\lambda), \quad \Phi_\lambda(x,y,t_0) = \frac{\lambda}{2} |x-y|^2.
\end{equation}
In particular, it is necessary to estimate the deviation of $\Phi_\lambda$ from its value at $t_0$, as well as measuring, in terms of $\lambda$, the largest interval $(t_0 - h_\lambda, t_0 + h_\lambda)$ on which $\Phi_\lambda(\cdot,t)$ remains smooth. The latter issue is resolved through a careful analysis of the characteristic system \eqref{E:homogchars}, taking advantage of the fact that $W$ is scalar and a change of variables in time. As a result, the assumptions on $H$ are substantially more complicated than in the spatially homogenous case, and more regularity is required for the paths. Rather than listing the assumptions in full generality, it is more instructive to consider certain ``model'' examples.

{\it Separated potential:} If, for some smooth $f: \RR^d \to \RR$, 
\[
	H(p,x) = \frac{1}{2} |p|^2 - f(x),
\]
then the comparison principle holds as long as $W \in C^\alpha$ with $\alpha \ge \alpha_0$, where $\alpha_0$ depends on the regularity of $f$. For instance, if $f \in C^2_b(\RR^d)$, then $\alpha_0 = \frac25$, while if $f \in C^3_b(\RR^d)$, $\alpha_0 = \frac14$. Both settings include the case of Brownian motion, or indeed any geometric rough path with the regularity discussed in this paper, although the proof of the comparison principle does not use the theory of rough paths, since the Hamiltonian and path are scalar.

{\it Linear growth:} If, for some $a \in C^3_b(\RR^d)$,
\[
	H(p,x) = a(x) \pars{ |p|^2 + 1}^{1/2},
\]
then the comparison principle holds for all continuous $W$.

The second-order case presents yet more challenges, since the solution of \eqref{E:specificdoubledeq} is used in the application of Lemma \ref{L:sums}, and it is not immediately clear that this yields useful information for the comparison principle. Resolving these difficulties is the subject of a future work by the author \cite{SeAppear}.

{\it Uniformly convex with homogenous structure:} The convexity of the Hamiltonians in the previous two examples (which, for the latter, holds only when $a > 0$) plays no role, and indeed, each of the two model cases above can be generalized to cover nonconvex Hamiltonians which exhibit similar structure and growth for large $p$. An alternative approach is to exploit convexity and adapt the proof of the comparison principle to the specific shape of the Hamiltonian. More precisely, the idea is to abandon the use of the test function \eqref{E:stationarydist} entirely, and instead search directly for a stationary, smooth solution of \eqref{E:smoothdoubledeq}. In general, this is not possible unless $H$ is scalar, uniformly convex in the gradient variable, and has some additional regularity and structure. Even then, the stationary solution, which is constructed through a variational formula, is only smooth in a neighborhood of the diagonal, although this is sufficient to prove the comparison principle. 

A class of Hamiltonians for which this strategy may be effectively carried out is given, for $q \ge 2$, $C \ge 1$, and $g \in C^2_b(\RR^d; S^d)$ satisfying $C\nv I_d \le g \le C I_d$, by
\[
	H(p,x) = \langle g(x)p, p \rangle^{q/2}.
\]
Because the approach more closely resembles that of the $x$-independent case, the comparison principle holds for all continuous paths, and moreover, the proof can be adapted to show that \eqref{E:eq} is stable with respect to uniform convergence in $W$ with explicit error estimates. Details when $q = 2$ may be found in \cite{FGLS} in the first and second-order cases, and in \cite{Se} for arbitrary $q$ in for first-order equations, where the theory is applied to various homogenization problems. Just as for the previous examples, the precise form of the Hamiltonian above is not important, and the arguments can be applied to Hamiltonians satisfying more general assumptions, which is the subject of a forthcoming work by Lions and Souganidis \cite{LSappear}.
 
\subsection{Multiple paths and $x$-dependent $H$} In this setting, which is treated by Lions and Souganidis \cite{LSappear}, the components of $H$ are assumed to have either a separated-potential or linear-growth structure as in the scalar case above, and $W$ is assumed to be a Brownian motion. The argument is much more involved, since it necessary to directly study the rough (or stochastic) system of characteristics \eqref{E:chars}.

\end{document}